\documentclass{amsart}
\usepackage{tikz,amssymb}
\usetikzlibrary{positioning}
\newtheorem{theorem}{Theorem}
\newtheorem{lemma}{Lemma}
\newtheorem{proposition}{Proposition}

\theoremstyle{definition}
\newtheorem{definition}{Definition}

\newtheorem{remark}{Remark}
\newcommand{\dfrt}{T_{dpr}}
\newcommand{\cfdvs}{\mathbf{VectFin}}

\newcommand{\mys}{\operatorname{ss}}
\newcommand{\myt}{\operatorname{st}}
\title{$R$-matrix knot invariants and triangulations}
\author{R. M.  Kashaev}
\thanks{The work is supported in part by the Swiss National Science
  Foundation.}
\address{Universit\'e de Gen\`eve,
Section de math\'ematiques,
2-4, rue du Li\`evre,
CP 64,
1211 Gen\`eve 4, Suisse
}
\date{February 9, 2010}
\email{rinat.kashaev@unige.ch}
\begin{document}
\begin{abstract}
The construction of quantum knot invariants from solutions of the Yang--Baxter equation ($R$-matrices) is reviewed with the emphasis on a class of $R$-matrices admitting an interpretation in intrinsically three-dimensional terms.
\end{abstract}
\maketitle
\section*{Introduction}

The Yang--Baxter equation \cite{Yang,Baxter} is the central tool for studying exactly solvable models of quantum field theory and statistical mechanics \cite{Baxter,Faddeev} and in the theory of quantum groups \cite{Drinfeld}. Its significance in knot theory and topology of three-manifolds comes from the fact that its solutions, called $R$-matrices, are the principal ingredients in the constructions of quantum knot invariants \cite{Jones,Turaev1}.

In this paper, we review the construction of quantum knot invariants by using $R$-matrices, putting the emphasis  on a particular class of $R$-matrices  which admit an interpretation in intrinsically three-dimensional terms. We show that for any finite dimensional Hopf algebra with invertible antipode, there exists a canonical $R$-matrix such that the associated knot invariant can be given an interpretation in terms of the combinatorics of ideal triangulations of knot complements. On the algebraic level, our construction corresponds to a canonical algebra homomorphism of the Drinfeld double of a Hopf algebra into the tensor product of the Heisenberg doubles of the same Hopf algebra and its dual Hopf algebra \cite{kash1}.

The organization of the paper is as follows. In Section~\ref{Sec:1}, we review the construction of knot invariants by using $R$-matrices. From one hand side, following the modern tendency, instead of finite dimensional vector spaces, we work in a more general framework of monoidal categories. On the other hand side, unlike the common practice of starting from a quasi-triangular Hopf algebra or (and) a braided monoidal category of its representations, we start from a particular $R$-matrix without precising its algebraic origin. In order to be able
to construct a knot invariant, we need only one extra condition on the $R$-matrix which we call rigidity. In the case of an $R$-matrix acting in a finite dimensional vector space, the rigidity is nothing else but the invertibility of the partially transposed matrix. Such approach first was suggested by N.~Reshetikhin in \cite{Reshetikhin}. In our mind, it is a direct and economical way of constructing a knot invariant from a given $R$-matrix, without preliminary study of the underlying quasi-triangular Hopf algebra and its representation theory. This approach is especially useful in the case of $R$-matrices whose algebraic nature is not known or poorly studied, for example, some of the two-dimensional $R$-matrices in the classification list of J. Hietarinta  \cite{Hietarinta}. In Section~\ref{Sec:2}, we introduce the notion of a rigid $T$-matrix and describe a canonical construction of a rigid $R$-matrix from a rigid $T$-matrix. Together with the result of the previous section, we are able to associate a knot invariant to any rigid $T$-matrix. Finally, in Section~\ref{Sec:3}, we associate to any tangle diagram a $3$-dimensional $\Delta$-complex in the sense of A.~Hatcher \cite{Hatcher}, and interpret the knot invariants associated with finite dimensional Hopf algebras in terms of state sums over $3$-dimensional $\Delta$-complexes.
\subsection*{Acknowledgements}
The author would like to thank N. Reshetikhin and A. Virelizier for interesting discussions.
\section{$R$-matrices and knot invariants}\label{Sec:1}
\subsection{Notation}
In this paper, the term monoidal category means strict monoidal category.
For notational simplification, we shall suppress the tensor product symbol but always keep the usual composition symbol $\circ$, and  to avoid writing many parentheses,
we shall assume that the tensor product takes precedence over the usual composition of morphisms. We shall also write $X$ instead of $\operatorname{id}_X$, provided the context permits to avoid confusion. The price paid for this notation is that powers of an endomorphism $f\colon X\to X$ are now preceded by the composition symbol:
\[
f^{\circ k}=\underbrace{f\circ f\circ\cdots \circ f}_{k\ \mathrm{times}},\quad k\in\mathbb{Z}_{\ge 0}.
\]
\subsection{Dual objects in monoidal categories}
In a monoidal category, a \emph{duality} is a quadruple
 $\langle X,Y,\eta,\epsilon\rangle$ consisting of two objects $X$, $Y$ and two morphisms
 \[
 \eta\colon \mathbb{I}\to YX,\quad
 \epsilon\colon XY\to \mathbb{I}
 \] such that
\begin{equation}\label{E:first}
\epsilon X\circ X\eta=X
,\quad Y\epsilon\circ\eta Y=Y.
\end{equation}
In this situation,  $X$ is called \emph{left dual} of $Y$, and $Y$ is called \emph{right dual} of $X$. An object $X$ is called \emph{rigid} if it is a right and a left dual. A rigid object is called \emph{strongly rigid} (or \emph{s-rigid}) if its left dual is isomorphic to its right dual.

For two dualities $p=\langle X,Y,\eta,\epsilon\rangle$, $q=\langle U,V,\eta',\epsilon'\rangle$, we define their product
\[
pq=\langle XU,VY,V\eta U\circ\eta',\epsilon\circ X\epsilon' Y\rangle.
\]

For a given duality $p=\langle X,Y,\eta,\epsilon\rangle$ and any objects $U$ and $V$, there are canonical bijections
\[
a_p\colon \operatorname{Hom}(XU,V)\to \operatorname{Hom}(U,YV),\quad
b_p\colon \operatorname{Hom}(UY,V)\to \operatorname{Hom}(U,VX)
\]
defined by the formulae
\[
a_p(f)=Yf\circ\eta U,\quad b_p(g)=gX\circ U\eta,\quad f\in\operatorname{Hom}(XU,V),\ g\in\operatorname{Hom}(UY,V),
\]
their inverses being
\[
a_p^{-1}(f)=\epsilon V\circ Xf,\quad b_p^{-1}(g)=V\epsilon\circ g Y,\quad f\in\operatorname{Hom}(U,YV),\ g\in \operatorname{Hom}(U,VX).
\]
 We shall use particular compositions of these bijections
\[
l_p=a_p\circ b_p^{-1}\colon \operatorname{Hom}(XU,VX)\to \operatorname{Hom}(UY,YV)
\]
and
\[
r_p=l_p^{-1}=b_p\circ a_p^{-1}\colon \operatorname{Hom}(UY,YV)\to \operatorname{Hom}(XU,VX).
\]
\begin{definition}
A morphism $f\colon XU\to VX$ (respectively $f\colon UX\to XV$) of a monoidal category is called $X$-\emph{left} (respectively $X$-\emph{right}) \emph{isomorphism}, if $X$ is a left (respectively right) dual, and for a duality $p=\langle X, Y,\eta,\epsilon\rangle$ (respectively
\(
 p=\langle Y, X,\eta,\epsilon\rangle
 \))
 the morphism $l_p(f)$ (respectively $r_p(f)$)
is an isomorphism.
\end{definition}
\begin{lemma}\label{Lem:iff}
Let $X$ be a left dual and $Y$ a right dual objects of a monoidal category. Then, a morphism $f\colon XY\to YX$ is $X$-left isomorphism if and only if it is $Y$-right isomorphism.
\end{lemma}
\begin{proof} Choose dualities
\(
p=\langle X,U,\eta,\epsilon\rangle
\) and
\(
 q=\langle V, Y,\eta',\epsilon'\rangle.
 \)
 If $f\colon XY\to YX$ is a $X$-left isomorphism, i.e. $l_p(f)$ is invertible, then $r_{pq}((l_p(f))^{\circ-1})$ is inverse of $r_q(f)$, i.e. $f$ is a $Y$-right isomorphism. Similarly, if $f$ is a $Y$-right isomorphism, i.e. $r_q(f)$ is invertible, then
 $l_{qp}((r_q(f))^{\circ-1})$ is inverse of $l_p(f)$, i.e. $f$ is a $X$-left isomorphism.
\end{proof}
\subsection{Rigid $R$-matrices}
\begin{definition}
An \emph{ $R$-matrix} is a triple $(C,X,\rho)$ given by a monoidal category $C$, an object $X$ in $C$, and an isomorphism  $\rho\colon X^{2}\to X^{2}$ satisfying the \emph{Yang--Baxter equation}
\[
\rho X\circ X\rho\circ\rho X=X\rho\circ\rho X\circ X\rho.
\]
\end{definition}
The morphism $\rho$ itself will also be informally called $R$-matrix, if the underlying triple is clear from the context.
\begin{lemma}\label{Rem:left-right-rigid}
 Let $(C,X,\rho)$ be an $R$-matrix. If $\rho^{\circ\pm 1}$ are $X$-left or $X$-right isomorphisms, then $X$ is a s-rigid object.
 \end{lemma}
 \begin{proof}
 Assume that $\rho^{\circ\pm 1}$ are $X$-left isomorphisms. Choosing a duality $p=\langle X,Y,\eta,\epsilon\rangle$, the morphisms
\[
\bar\eta=\rho_-^{\circ-1}\circ\eta\colon \mathbb{I}\to XY,\quad \bar\epsilon=\epsilon\circ\rho_+^{\circ-1}\colon YX\to \mathbb{I},\quad \rho_\pm=l_p(\rho^{\circ\pm1}),
\]
satisfy the identities
\[
\bar\epsilon Y\circ Y\bar\eta=Y,\quad X\bar\epsilon\circ\bar\eta X=X
\]
which mean that we have another duality
\[
\bar p=\langle Y,X,\bar\eta,\bar\epsilon\rangle,
\]
i.e. $X$ is a right dual of $Y$, and thus $X$ is s-rigid.
The case when $\rho^{\circ\pm 1}$ are $X$-right isomorphisms is similar.
\end{proof}
\begin{definition}
An $R$-matrix $(C,X,\rho)$ is called \emph{rigid} if $\rho^{\circ\pm1}$ are $X$-left isomorphisms (and by Lemmas~\ref{Lem:iff} and \ref{Rem:left-right-rigid} they are also $X$-right isomorphisms).
\end{definition}
\begin{remark}
 The rigidity of an $R$-matrix $(C,X,\rho)$, where $C$ is the category of finite dimensional vector spaces, is equivalent to its non-degenerateness in the sense of \cite{Reshetikhin}.
\end{remark}

\subsection{The category of directed planar ribbon tangles}
A \emph{directed tangle} is a smooth embedding $f\colon (M,\partial M)\to (\mathbb{R}^2\times [0,1],\mathbb{R}^2\times \partial[0,1])$ where $M$ is a compact oriented one-dimensional manifold.
A \emph{directed tangle diagram} is the image of a directed tangle under the projection
$\mathbb{R}^2\times [0,1]\to \mathbb{R}\times [0,1]$, $(x,y,t)\mapsto(x,t)$, provided the tangle is in general position with respect to the projection, and with the additional information of over- and under-crossings as with knot diagrams. A \emph{connected component} of a directed tangle diagram is the image of a connected component of the corresponding  directed tangle. A \emph{directed planar ribbon tangle}, to be referred later as \emph{dpr-tangle}, is an equivalence class of directed tangle diagrams with respect to the equivalence relation generated by isotopies of the strip $\mathbb{R}\times [0,1]$ and the oriented Reidemeister moves of types II and III. In other words, a dpr-tangle is a directed tangle diagram considered up to regular isotopies \cite{Kauffman}.

The set of dpr-tangles forms the set of morphisms of a  monoidal category $\dfrt$ defined in a similar way  as for ordinary directed tangles \cite{Turaev}. The set of objects $\operatorname{Ob}\dfrt$ are finite sequences (including the empty one) of pluses and minuses.
For a dpr-tangle $x$, the elements of the sequence $\operatorname{dom}(x)$  (respectively $\operatorname{cod}(x)$) are in bijective correspondence with the points of the intersection set $x\cap(\mathbb{R}\times\{0\})$ (respectively $x\cap(\mathbb{R}\times\{1\})$) with the order induced from that of $\mathbb{R}$, and for each intersection point the associated sign is plus, if the point is the image of the initial (respectively terminal) point of the corresponding oriented segment
of the directed tangle, and minus otherwise. For two composable dpr-tangles $x$ and $y$ (i.e. satisfying the condition $\operatorname{cod}(x)=\operatorname{dom}(y)$) their composition is obtained by choosing representatives of $x$ and $y$ so that the projection images in the real line of the intersections $x\cap(\mathbb{R}\times\{1\})$ and $y\cap(\mathbb{R}\times\{0\})$ coincide, and then by identifying
the boundary component  $\mathbb{R}\times\{1\}$  of the strip containing $x$ with the boundary component  $\mathbb{R}\times\{0\}$ of the strip containing $y$, thus forming a new strip containing a representative of the composed dpr-tangle $y\circ x$. The tensor product $xy$ of two dpr-tangles $x$ and $y$ is obtained by choosing representatives of $x$ and $y$ so that the representative of $x$ is located to the left of the representative of $y$, and then by taking their union as a representative of $xy$.

As a monoidal category, the category $\dfrt$ is generated by the following dpr-tangles:
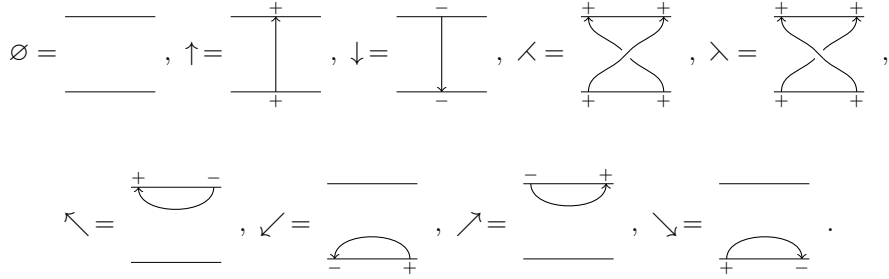
\[
\varnothing=\begin{tikzpicture}[baseline=(x.south)]
 \draw[very thin] (-.1,0)--(1.1,0);
\draw[very thin](-.1,1)--(1.1,1);
\node (x) at (.5,.5) [inner sep=.5mm,circle] {};
 \end{tikzpicture}\ ,\
 \uparrow\,=\begin{tikzpicture}[baseline=(x.south)]
 \draw[very thin] (-.1,0)--(1.1,0);
\draw[very thin] (-.1,1)--(1.1,1);
\node (x) at (.5,.5) [inner sep=.5mm,circle] {};
 \draw[->] (.5,0)--(.5,1);
 \node at (.5,-.12){\tiny $+$};
 \node at (.5,1.12){\tiny $+$};
 \end{tikzpicture}\ ,\
 \downarrow\,=\begin{tikzpicture}[baseline=(x.south)]
 \draw[very thin] (-.1,0)--(1.1,0);
\draw[very thin] (-.1,1)--(1.1,1);
\node (x) at (.5,.5) [inner sep=.5mm,circle] {};
 \draw[->] (.5,1)--(.5,0);
 \node at (.5,-.12){\tiny $-$};
 \node at (.5,1.12){\tiny $-$};
 \end{tikzpicture}\ ,\
\rightthreetimes=\begin{tikzpicture}[baseline=(x.south)]
\draw[very thin] (-.1,0)--(1.1,0);
\draw[very thin] (-.1,1)--(1.1,1);
\node (x) at (.5,.5) [inner sep=.5mm,circle] {};
 \draw[->] (x.north east) to [out=45,in=-90] (1,1);
 \draw (x.south west) to (x.north east);
 \draw (0,0) to [out=90,in=-135] (x.south west);

 \draw (1,0) to [out=90,in=-45] (x.south east);
 \draw[->] (x.north west) to [out=135,in=-90] (0,1);
 \node at (0,-.12){\tiny $+$};
 \node at (0,1.12){\tiny $+$};
 \node at (1,-.12){\tiny $+$};
 \node at (1,1.12){\tiny $+$};
 \end{tikzpicture}
 \ ,\
 \leftthreetimes=\begin{tikzpicture}[baseline=(x.south)]
 \draw[very thin] (-.1,0)--(1.1,0);
\draw[very thin] (-.1,1)--(1.1,1);
 \node (x) at (.5,.5) [inner sep=.5mm,circle] {};
 \draw[->] (x.north east) to [out=45,in=-90] (1,1);
 \draw (x.south east) to (x.north west);
 \draw (0,0) to [out=90,in=-135] (x.south west);

 \draw (1,0) to [out=90,in=-45] (x.south east);
 \draw[->] (x.north west) to [out=135,in=-90] (0,1);
 \node at (0,-.12){\tiny $+$};
 \node at (0,1.12){\tiny $+$};
 \node at (1,-.12){\tiny $+$};
 \node at (1,1.12){\tiny $+$};
 \end{tikzpicture}\ ,
 \]
 \[
 \nwarrow\,=\begin{tikzpicture}[baseline=(x.south)]
 \draw[very thin] (-.1,0)--(1.1,0);
\draw[very thin] (-.1,1)--(1.1,1);
\node (x) at (.5,.5) [inner sep=.5mm,circle] {};
 \draw[->] (1,1) to [out=-90,in=-90](0,1);
 \node at (0,1.12){\tiny $+$};
 \node at (1,1.12){\tiny $-$};
 \end{tikzpicture}\ ,\
 \swarrow\,=\begin{tikzpicture}[baseline=(x.south)]
 \draw[very thin] (-.1,0)--(1.1,0);
\draw[very thin] (-.1,1)--(1.1,1);
 \node (x) at (.5,.5){};
 \draw[->] (1,0) to [out=90,in=90](0,0);
 \node at (0,-.12){\tiny $-$};
 \node at (1,-.12){\tiny $+$};
 \end{tikzpicture}\ ,\
 \nearrow\,=\begin{tikzpicture}[baseline=(x.south)]
\draw[very thin] (-.1,0)--(1.1,0);
\draw[very thin] (-.1,1)--(1.1,1);
\node (x) at (.5,.5){};
 \draw[->] (0,1) to [out=-90,in=-90](1,1);
 \node at (0,1.12){\tiny $-$};
 \node at (1,1.12){\tiny $+$};
 \end{tikzpicture}\ ,\
 \searrow\,=\begin{tikzpicture}[baseline=(x.south)]
 \draw[very thin] (-.1,0)--(1.1,0);
\draw[very thin] (-.1,1)--(1.1,1);
\node (x) at (.5,.5){};
 \draw[->] (0,0) to [out=90,in=90](1,0);
 \node at (0,-.12){\tiny $+$};
 \node at (1,-.12){\tiny $-$};
 \end{tikzpicture}\ .
\]
For any dpr-tangle $x$, there are two functions on the set of its connected components
\[
\operatorname{wr}\colon \pi_0(x)\to\mathbb{Z},\quad \operatorname{wn}\colon \pi_0(x)\to2^{-1}\mathbb{Z},
\]
called respectively \emph{writhe} and \emph{winding number}. For a connected dpr-tangle $x$, these functions are calculated according to the formulae
\[
\operatorname{wr}(x)=\#(\rightthreetimes)-\#(\leftthreetimes),\quad 2\operatorname{wn}(x)=
\#(\nearrow)+\#(\swarrow)-\#(\searrow)-\#(\nwarrow),
\]
where $\#(y)$ denotes the number of generating elements of type $y$ in a given decomposition of a diagram representing $x$.
\begin{remark}
The usual category of framed directed tangles is a quotient category of $\dfrt$ obtained by adding the Reidemeister move of type $I'$.
\end{remark}
\subsection{Knot invariants from rigid $R$-matrices}
The following theorem is essentially a reformulation of a result of N. Reshetikhin (Theorem~1.1 of \cite{Reshetikhin}).
\begin{theorem}\label{Th:main}
Let $(C,X,\rho)$ be a rigid $R$-matrix. Fix a duality
\(
p=\langle Y,X,\eta,\epsilon\rangle.
\)
 Then,
for any integer $k$ there exists a unique functor $\Phi_{\rho,k}\colon \dfrt\to C$ such that
\[
\begin{tikzpicture}[baseline=(x.center)]
\draw[very thin] (-.1,0)--(1.1,0);
\draw[very thin] (-.1,1)--(1.1,1);
 \node (x) at (.5,.5) [inner sep=.5mm,circle] {};
 \draw[->] (x.north east) to [out=45,in=-90] (1,1);
 \draw (x.south west) to (x.north east);
 \draw (0,0) to [out=90,in=-135] (x.south west);

 \draw (1,0) to [out=90,in=-45] (x.south east);
 \draw[->] (x.north west) to [out=135,in=-90] (0,1);
 \end{tikzpicture}\mapsto\rho,\
 \begin{tikzpicture}[baseline=(x.center)]
 \draw[very thin] (-.1,0)--(1.1,0);
\draw[very thin] (-.1,1)--(1.1,1);
\node (x) at (.5,.5){};
 \draw[->] (1,1) to [out=-90,in=-90](0,1);
 \end{tikzpicture}\mapsto \eta,\
 \begin{tikzpicture}[baseline=(x.center)]
 \draw[very thin] (-.1,0)--(1.1,0);
\draw[very thin] (-.1,1)--(1.1,1);
 \node (x) at (.5,.5){};
 \draw[->] (1,0) to [out=90,in=90](0,0);
 \end{tikzpicture}\mapsto \epsilon,\
 \begin{tikzpicture}[baseline=(x.center)]
 \draw[very thin] (-.1,0)--(1.1,0);
\draw[very thin] (-.1,1)--(1.1,1);
\node (x) at (.5,.5){};
 \draw[->] (0,1) to [out=-90,in=-90](1,1);
 \end{tikzpicture}\mapsto \bar\eta,\
 \begin{tikzpicture}[baseline=(x.center)]
 \draw[very thin] (-.1,0)--(1.1,0);
\draw[very thin] (-.1,1)--(1.1,1);
 \node (x) at (.5,.5){};
 \draw[->] (0,0) to [out=90,in=90](1,0);
 \end{tikzpicture}\mapsto \bar\epsilon,
\]
where
\[
\bar\eta=Y\omega^{\circ k}\circ\rho_-^{\circ-1}\circ\eta,\quad \bar\epsilon=\epsilon\circ\rho_+^{\circ-1}\circ\omega^{\circ-k}Y,\quad \rho_\pm=r_p(\rho^{\circ\pm1}),
\]
\[
\omega=(\epsilon\circ\rho_-^{\circ-1})X\circ X(\rho_-^{\circ-1}\circ \eta).
\]
If, furthermore, there exists a morphism $\nu\colon X\to X$ such that
\[
X\nu\circ\rho^{\circ\pm 1}=\rho^{\circ\pm 1}\circ\nu X,\quad \omega=\nu^{\circ 2},
\]
then the parameter $k$ can take half-integer values, and the functor $\Phi_{-1/2,\rho}$ factors through the category of framed directed tangles.
\end{theorem}
\begin{proof}
The proof goes along a similar line of arguments as the proof of Theorem~1.1 of \cite{Reshetikhin}.
\end{proof}
\begin{theorem}\label{Th:second}
Let $x$ be a connected dpr-tangle with $\operatorname{dom}(x)=\operatorname{cod}(x)=(+)$. Then, the morphism
\begin{equation}\label{E:knot-invariant}
\Psi_\rho(x)=\omega^{\circ-(\operatorname{wr}(x)+(1+2k)\operatorname{wn}(x))/2}\circ \Phi_{\rho,k}(x)\colon X\to X
\end{equation}
is independent of $k$ and invariant under all Reidemeister moves of type I.
\end{theorem}
\begin{proof} The statement directly follows from Theorem~\ref{Th:main} and the formulae
\[
\Phi_{\rho,k}\left(\begin{tikzpicture}[baseline=(x.center)]
\draw[very thin] (-.1,0)--(1.1,0);
\draw[very thin] (-.1,1)--(1.1,1);
 \node (x) at (.5,.5) [inner sep=.5mm,circle] {};
 \draw (x.north east) to [out=45,in=90] (.8,.5);
 \draw (x.south west) to (x.north east);
 \draw (.5,0) to [out=90,in=-135] (x.south west);
 \draw (.8,.5) to [out=-90,in=-45] (x.south east);
 \draw[->] (x.north west) to [out=135,in=-90] (.5,1);
 \end{tikzpicture}\right)=\omega^{\circ-k},\quad
 \Phi_{\rho,k}\left(\begin{tikzpicture}[baseline=(x.center)]
\draw[very thin] (-.1,0)--(1.1,0);
\draw[very thin] (-.1,1)--(1.1,1);
 \node (x) at (.5,.5) [inner sep=.5mm,circle] {};
 \draw (x.north east) to [out=45,in=90] (.8,.5);
 \draw (x.south east) to (x.north west);
 \draw (.5,0) to [out=90,in=-135] (x.south west);
 \draw (.8,.5) to [out=-90,in=-45] (x.south east);
 \draw[->] (x.north west) to [out=135,in=-90] (.5,1);
 \end{tikzpicture}\right)=\omega^{\circ-1-k},
 \]
 \[
\Phi_{\rho,k}\left(\begin{tikzpicture}[baseline=(x.center)]
\draw[very thin] (-.1,0)--(1.1,0);
\draw[very thin] (-.1,1)--(1.1,1);
 \node (x) at (.5,.5) [inner sep=.5mm,circle] {};
 \draw (x.north west) to [out=135,in=90] (.2,.5);
 \draw (x.south west) to (x.north east);
 \draw (.5,0) to [out=90,in=-45] (x.south east);
 \draw (.2,.5) to [out=-90,in=-135] (x.south west);
 \draw[->] (x.north east) to [out=45,in=-90] (.5,1);
 \end{tikzpicture}\right)=\omega^{\circ1+k},\quad
 \Phi_{\rho,k}\left(\begin{tikzpicture}[baseline=(x.center)]
\draw[very thin] (-.1,0)--(1.1,0);
\draw[very thin] (-.1,1)--(1.1,1);
 \node (x) at (.5,.5) [inner sep=.5mm,circle] {};
 \draw (x.north west) to [out=135,in=90] (.2,.5);
 \draw (x.south east) to (x.north west);
 \draw (.5,0) to [out=90,in=-45] (x.south east);
 \draw (.2,.5) to [out=-90,in=-135] (x.south west);
 \draw[->] (x.north east) to [out=45,in=-90] (.5,1);
 \end{tikzpicture}\right)=\omega^{\circ k}.
 \]
 Notice that for such $x$ the sum $\operatorname{wr}(x)+\operatorname{wn}(x)$ is always even, so that the power of $\omega$ in \eqref{E:knot-invariant} is an integer.
\end{proof}
If we define a knot as a $(1,1)$-tangle, then Theorem~\ref{Th:second} implies that function $\Psi_\rho (x)$ is an invariant of the knot represented by  $x$.
\section{$R$-matrices from $T$-matrices}\label{Sec:2}
In this section, we define rigid $T$-matrices and, following the ideas of \cite{kash1}, describe a canonical construction of rigid $R$-matrices from rigid $T$-matrices.
\subsection{Rigid $T$-matrices}
\begin{definition}
A \emph{ $T$-matrix} is a quadruple $(C,X,\sigma,\tau)$ given by an involutive $R$-matrix $(C,X,\sigma)$, i.e.  satisfying the condition $\sigma=\sigma^{\circ-1}$, and an isomorphism  $\tau\colon X^2\to X^2$ satisfying the equations
\begin{equation}\label{E:naturality}
\tau X\circ X\sigma\circ\sigma X=X\sigma\circ\sigma X\circ X\tau,
\end{equation}
\begin{equation}\label{E:pentagon}
\tau X\circ X\sigma\circ\tau X=X\tau\circ\tau X\circ X\tau.
\end{equation}
\end{definition}
As in the case of $R$-matrices, the morphism $\tau$ itself will informally be called $T$-matrix, if the underlying quadruple is clear from the context.
\begin{remark}\label{rem:tay-tau-inverse-symmetry}
It is easily verified that if $(C,X,\sigma,\tau)$ is a $T$-matrix, then $(C,X,\sigma,\tau^{\circ-1})$ is also a $T$-matrix. This symmetry will allow to simplify some proofs below.
\end{remark}
\begin{definition}
A $T$-matrix $(C,X,\sigma,\tau)$ is called \emph{rigid} if $(C,X,\sigma)$ is a rigid $R$-matrix and $\tau^{\circ\pm 1}$ are $X$-left isomorphisms (and, by Lemma~\ref{Lem:iff}, also $X$-right isomorphisms).
\end{definition}
\begin{lemma}\label{Lem:double-rigidity}
 For a rigid $T$-matrix $(C,X,\sigma,\tau)$ and a duality $p=\langle Y,X,\eta,\epsilon\rangle$,  the morphisms $(r_p(\tau^{\circ\pm1}))^{\circ-1}$ are $X$-left and $Y$-right isomorphisms.
\end{lemma}
\begin{proof} Due to Lemma~\ref{Rem:left-right-rigid} and Remark~\ref{rem:tay-tau-inverse-symmetry}, it is enough to prove that $\tau_+=(r_p(\tau))^{\circ-1}$ is a $Y$-right isomorphism. Defining another duality
\[
q=\langle X,Y,\bar\eta=(r_p(\sigma))^{\circ-1}\circ\eta,\bar\epsilon=\epsilon\circ (r_p(\sigma))^{\circ-1}\rangle,
\]
the morphisms
\[
\alpha=a_q^{-1}(l_q(\tau)X\circ X\tau_+\circ\sigma Y\circ Xr_{pq}(\tau_+)\circ(r_p(\tau)\circ\bar\eta)X)
\]
and
\[
\beta=a_p((\epsilon\circ l_q(\tau))X\circ X\tau_+\circ\sigma Y\circ Xr_{pq}(\tau_+)\circ r_p(\tau)X)
\]
are such that
\[
\beta\circ r_q(\tau_+)=r_q(\tau_+)\circ\alpha=X^2
\]
which, by considering the product $\beta\circ r_q(\tau_+)\circ\alpha$, imply that $\alpha=\beta=(r_q(\tau_+))^{\circ-1}$. Thus, $\tau_+$ is a $Y$-right isomorphism.
\end{proof}
\subsection{Rigid $R$-matrices from rigid $T$-matrices}
\begin{theorem}
For a rigid $T$-matrix $(C,X,\sigma,\tau)$ and a duality $p=\langle Y,X,\eta,\epsilon\rangle$, the triple
\[
(C,XY,\rho_\tau),\quad \rho_\tau=X(r_p(\tau))^{\circ-1}Y\circ \tau r_p(r_p(\tau))\circ Xr_p(\tau^{\circ-1})Y,
\]
is a rigid $R$-matrix with self-dual object $XY$.
\end{theorem}
\begin{proof}
For the inverse of $\rho_\tau$, we have the formula
\[
\rho_\tau^{\circ-1}=\rho_{\tau^{\circ-1}},
\]
while the Yang--Baxter equation
\[
\rho_\tau XY\circ XY\rho_\tau\circ \rho_\tau XY=XY\rho_\tau\circ \rho_\tau XY\circ XY\rho_\tau
\]
is a consequence of eight Pentagon relations
\[
\tau X\circ X\sigma\circ\tau X=X\tau\circ\tau X\circ X\tau,\quad
\tau' Y\circ Y\tau'\circ\tau' Y=Y\tau'\circ\sigma' Y\circ Y\tau',
\]
\[
\tau Y\circ X\hat\sigma\circ\hat\tau X=X\hat\tau\circ\hat\tau X\circ Y\tau,\quad
\tau' X\circ Y\check\tau\circ\check\tau Y=Y\check\tau\circ\check\sigma Y\circ X\tau',
\]
\[
 \hat\tau X\circ Y\tau\circ \check\tau X=X\check\tau\circ\sigma Y\circ X\hat\tau,\quad
 \check\tau Y\circ X\sigma'\circ \hat\tau Y=Y\hat\tau\circ\tau' X\circ Y\check\tau,
\]
\[
 \check\tau X\circ X\check\sigma\circ \tau Y=Y\tau\circ\check\tau X\circ X\check\tau,\quad
 \hat\tau Y\circ Y\hat\tau\circ \tau' X=X\tau'\circ\hat\sigma Y\circ Y\hat\tau
\]
where we use the notation
\[
\hat\tau=r_p(\tau^{\circ-1}),\quad\check\tau=(r_p(\tau))^{\circ-1},\quad\tau'=r_p(r_p(\tau)),
\]
\[
\hat\sigma=r_p(\sigma),\quad\check\sigma=\hat\sigma^{\circ-1},\quad\sigma'=r_p(r_p(\sigma)).
 \]
In fact, all of these Pentagon relations are equivalent to each other.
If we define another duality $q=\langle X,Y,\check\sigma\circ\eta,\epsilon\circ\check\sigma\rangle$, then
the object $XY$ is selfdual as it enters the duality
\[
qp=\langle XY,XY,X(\check\sigma\circ \eta)Y\circ\eta,\epsilon\circ\check\sigma\circ X\epsilon Y\rangle.
\]
Rigidity of $\rho_\tau$ is equivalent to invertibility of the morphisms $r_{qp}(\rho_\tau^{\circ\pm1})$. We have explicitly
\[
r_{qp}(\rho_\tau)=Xl_q(\tau)Y\circ r_q(\check\tau)(\tau')^{\circ-1}\circ Xr_p(\tau)Y,
\]
\[r_{qp}(\rho_\tau^{\circ-1})=Xl_q(\tau^{\circ-1})Y\circ r_q(\hat\tau^{\circ-1})\tau'\circ X\hat\tau Y
\]
which are invertible by Lemma~\ref{Lem:double-rigidity}.
\end{proof}
\subsection{$T$-matrices from Hopf objects}
Recall that a \emph{symmetric monoidal category} is a monoidal category $C$ with a natural isomorphism $s\colon\otimes\to \otimes^{\mathrm{op}}$ satisfying the equations
\[
s_{X,YZ}=Ys_{X,Z}\circ s_{X,Y}Z,\quad s_{XY,Z}=s_{X,Z}Y\circ Xs_{Y,Z},\quad s_{Y,X}s_{X,Y}=XY.
\]
\begin{definition}
A \emph{Hopf object} in a symmetric monoidal category $C$ is a six-tuple $(X,\nabla,\Delta,\eta,\epsilon,\gamma)$ consisting of an object $X$ of $C$ and five morphisms:  the \emph{product} $\nabla\colon X^2\to X$, the \emph{co-product} $\Delta\colon X\to X^2$, the \emph{unit} $\eta\colon\mathbb{I}\to X$, the \emph{co-unit} $\epsilon\colon X\to \mathbb{I}$, and the \emph{antipode} $\gamma\colon X\to X$ such that
\[
\nabla\circ\nabla X=\nabla\circ X\nabla ,\quad \nabla\circ X\eta=\nabla\circ\eta X=X,
\]
i.e. the triple $(X,\nabla, \eta)$ is a monoid in $C$,
\[
X\Delta\circ\Delta=\Delta X\circ\Delta,\quad X\epsilon\circ\Delta=\epsilon X\circ\Delta=X,
\]
i.e. the triple $(X,\Delta,\epsilon)$ is a co-monoid in $C$,
\[
\epsilon\circ\nabla=\epsilon^2,\quad \Delta\circ\eta=\eta^2,\quad
\epsilon\circ\eta=\mathbb{I},
\]
\[
\nabla\circ X\gamma\circ\Delta=\nabla\circ \gamma X\circ\Delta=\eta\circ\epsilon,
\]
and
\[
\Delta\circ\nabla=\nabla^2\circ Xs_{X,X}X\circ\Delta^2.
\]
\end{definition}
\begin{proposition}
Let $(X,\nabla,\Delta,\eta,\epsilon,\gamma)$ be a Hopf object in a symmetric monoidal category $C$.
Then, the quadruple
\begin{equation}\label{E:hopf-t-matrix}
(C,X,\sigma,\tau),\quad
\sigma=s_{X,X},\quad  \tau=\sigma\circ X\nabla\circ\Delta X
\end{equation} is a $T$-matrix. Moreover, if $X$ is a left or right dual, and the antipode is invertible, then this $T$-matrix is rigid.
\end{proposition}
\begin{proof}
Verification of the equations~\eqref{E:naturality} and \eqref{E:pentagon} is straightforward. Assume that $X$ is a left or right dual, and the antipode is invertible.
Then, it is easily seen that the involutive $R$-matrix $(C,X,\sigma)$ is rigid so that $X$ is s-rigid.
Fixing a duality $p=\langle Y,X,\eta,\epsilon\rangle$,  we also have
 another duality
 \[
 q=\langle X,Y,s_{X,Y}\circ\eta,\epsilon\circ s_{X,Y}\rangle.
 \]
  Now,
$\tau^{\circ\pm1}$ are $X$-right isomorphisms due to the identities
\[
(r_p(\tau))^{\circ-1}=l_q(\tau_{-1}\circ\sigma),\quad
(r_p(\tau^{\circ-1}))^{\circ-1}=l_q(\sigma\circ\tau_2 )
\]
where
\[
\tau_k=X\nabla\circ X\gamma^{\circ k}X\circ\Delta X,\quad k\in\mathbb{Z}.
\]
\end{proof}
\section{Knot invariants and $\Delta$-complexes}\label{Sec:3}
The results of the previous sections permit us to lift certain knot invariants, associated to finite dimensional Hopf algebras with invertible antipode, into the combinatorial setting of ideal triangulations of knot complements. Such a lifting is based on the geometrical interpretation of the $T$-matrices in terms of tetrahedra.
\subsection{$\Delta$-complexes}
Let
\[
\Delta^n=\{(t_0,t_1,\ldots,t_n)\in [0,1]^{n+1}\vert\ t_0+t_1+\cdots+t_n=1\},\quad n\ge 0,
\]
be the standard $n$-simplex with the face inclusion maps
\[
\delta_m\colon \Delta^{n}\to\Delta^{n+1},\quad 0\le m\le n+1,
\] defined by the formula
\[
\delta_m(t_0,\ldots,t_{n})=(t_0,\ldots, t_{m-1},0,t_{m},\ldots,t_n).
\]
A (simplicial) cell in a topological space $X$ is a continuous map
\[
f\colon \Delta^n\to X
\]
 such that the restriction of $f$ to the interior of $\Delta^n$ is an embedding. On the set of cells $\Sigma(X)$ we have the dimension function
 \[
 d\colon\Sigma(X)\to\mathbb{Z},\quad (f\colon \Delta^n\to X) \mapsto n.
 \]
Following \cite{Hatcher}, we define a $\Delta$-complex structure on a topological space $X$ as a pair $(\Delta(X), \partial)$, where $\Delta(X)\subset \Sigma(X)$ and $\partial$ is a set of maps
\[
\partial=\left\{\left.\partial_n\colon d\vert_{\Delta(X)}^{-1}(\mathbb{Z}_{\ge \max(1,n)})\to d\vert_{\Delta(X)}^{-1}(\mathbb{Z}_{\ge {n-1}})\ \right\vert\ n\ge0 \right\}
\]
such that:
\begin{itemize}
\item[(i)] each point of $X$ is in the image of exactly one restriction, $\alpha\vert_{\operatorname{int}(\Delta^{n})}$ for $\alpha\in \Delta(X)$;

\item[(ii)] $\alpha\circ\delta_m=\partial_m(\alpha)$;

\item[(iii)]a set $A\subset X$ is open iff $\alpha^{-1}(A)$ is open for each $\alpha\in \Delta(X)$.
\end{itemize}Notice that any $\Delta$-complex is a CW-complex.
We denote by $\Delta^n(X)=\Delta(X)\cap d^{-1}(n)$ the set of $n$-dimensional cells of a $\Delta$-complex $X$.
\begin{definition}
A \emph{combinatorial $\Delta$-complex} is a triple $(S,d,\partial)$ consisting of a set $S$, a map
$d\colon S\to\mathbb{Z}_{\ge0}$, and a family of maps
\[
\partial=\{\partial_n\colon d^{-1}(\mathbb{Z}_{\ge\operatorname{max}(1,n)})\to
d^{-1}(\mathbb{Z}_{\ge n-1})\vert\ n\in\mathbb{Z}_{\ge0}\}
\]
such that
\[
\partial_i(d^{-1}(n))\subset d^{-1}(n-1),\quad 0\le i\le n,\quad n\ge1,
\]
and
\[
\partial_i\circ\partial_{j+1}=\partial_j\circ\partial_i,\quad i\le j\,.
\]
A \emph{morphism} between two $\Delta$-complexes $(S,d,\partial)$ and $(S',d',\partial')$ is a map of sets $f\colon S\to S'$ such that $d'\circ f=d$ and $f\circ\partial_i=\partial_i'\circ f$, $i\ge 0$.
\end{definition}
For any $\Delta$-complex $X$, the triple $(\Delta(X),d\vert_{\Delta(X)},\partial)$ is a combinatorial $\Delta$-complex, and it is clear that any combinatorial $\Delta$-complex $(S,d,\partial)$ can be realized this way. Indeed, we define a topological space $X$ as the quotient space $\tilde X/\mathcal{R}$ where
\[
\tilde X=\bigsqcup_{n\ge 0} \Delta^n\times d^{-1}(n),
\]
and $\mathcal{R}$ is the equivalence relation generated by the relations
\[
(\delta_i(x),\alpha)\stackrel{\mathcal{R}}{\sim} (x,\partial_i(\alpha)),\quad 0\le i\le n,\quad x\in\Delta^{n-1},\quad \alpha\in d^{-1}(n),\quad n\ge1.
\]
Any element $\alpha\in d^{-1}(n)\subset S$ corresponds to an $n$-cell
\[
f_\alpha\colon \Delta^n\to X,\quad f_\alpha(x)=[x,\alpha],
\]
and, defining $\partial_i(f_\alpha)=f_{\partial_i(\alpha)}$, we verify the property (ii):
\[
f_\alpha\circ\delta_i(x)=[\delta_i(x),\alpha]=[x,\partial_i(\alpha)]=f_{\partial_i(\alpha)}(x)=
\partial_i(f_\alpha)(x).
\]
The same space can be constructed as a CW-complex by noting that
\[
\cup_{i=0}^n\delta_i(\Delta^{n-1})=\partial(\Delta^n),\quad n\ge1.
\]
We define the $0$-skeleton $X^0$ to be the discrete space $d^{-1}(0)$, and for $n\ge1$, we define recursively the $n$-skeleton $X^n$ as the space obtained by attaching the space $\Delta^n\times d^{-1}(n)$ on the $n-1$-skeleton $X^{n-1}$ along the boundary
\[
\partial(\Delta^n\times d^{-1}(n))=\partial(\Delta^n)\times d^{-1}(n)
\]
via the map
\[
\phi_n\colon \partial(\Delta^n\times d^{-1}(n))\to X^{n-1}, \quad \phi_n(\delta_i(x),\alpha)=\Phi_{n-1}(x,\partial_i(\alpha)),
\]
where $x\in \Delta^{n-1}$, $0\le i\le n$, and
\[
\Phi_{n-1}\colon \Delta^{n-1}\times d^{-1}(n-1)\to X^{n-1}
\]
is the characteristic map of the $n-1$-cells.

\subsection{Presentations of finite dimensional $\Delta$-complexes}
A combinatorial $\Delta$-complex $(S,d,\partial)$ is called \emph{finite dimensional} if there exists an integer $n\ge 0$ such that $d^{-1}(k)=\emptyset$ for all $k>n$. In this case, the \emph{dimension} of the complex is defined as the unique integer $\operatorname{dim}(S,d,\partial)$ such that $d^{-1}(\operatorname{dim}(S,d,\partial))\ne \emptyset$, and
$d^{-1}(k)=\emptyset$ for all $k>\operatorname{dim}(S,d,\partial)$.
\begin{proposition}\label{Prop:pres}
Let $A$, $B$ be two sets and $\{ f_i\colon A\to B\vert \ 0\le i\le n\}$ a set of $n+1$ maps from $A$ to $B$. Then, there exists a unique, up to a unique isomorphism of $\Delta$-complexes, $n$-dimensional combinatorial $\Delta$-complex $(S,d,\partial)$ such that
\begin{equation}\label{E:pres}
d^{-1}(n)=A,\quad d^{-1}(n-1)=B,\quad \partial_i\vert_{d^{-1}(n)}=f_i,\ 0\le i\le n,
\end{equation}
and it is universal in the sense that for any other  $n$-dimensional combinatorial $\Delta$-complex $(S',d',\partial')$, having the above properties, there exists a unique morphism
\[
f\colon (S,d,\partial)\to(S',d',\partial')
\]
 which
is identity on $d^{-1}\{n,n-1\}$.
\end{proposition}
\begin{proof}
We define $d^{-1}(m)$ for $m\in\{n,n-1\}$ and the maps $\partial_i\vert_{d^{-1}(n)}$ by using equations \eqref{E:pres}. Then, for $m< n-1$, we recursively define
\[
d^{-1}(m)=(d^{-1}(m+1)\times\{0,\ldots,m+1\})/\mathcal{R}_m
\]
where the equivalence relation $\mathcal{R}_m$ is generated by the relations
\[
(\partial_{j+1}(\alpha),i)\stackrel{\mathcal{R}_m}{\sim}(\partial_{i}(\alpha),j),\quad 0\le i\le j\le m+1,\quad \alpha\in d^{-1}(m+2).
\]
Denoting by $[\alpha,i]$ the $\mathcal{R}_m$-equivalence class of $(\alpha,i)$, we define
\[
\partial_i(\alpha)=[\alpha,i],\quad 0\le i\le m+1,\quad \alpha\in d^{-1}(m+1).
\]
Then, for any $\alpha\in d^{-1}(m+2)$ and $ 0\le i\le j\le m+1$, we verify that
\[
\partial_i\circ\partial_{j+1}(\alpha)=[\partial_{j+1}(\alpha),i]=[\partial_{i}(\alpha),j]=
\partial_j\circ\partial_{i}(\alpha).
\]
In this way, we come to an $n$-dimensional combinatorial $\Delta$-complex $(S,d,\partial)$, where
 $S=\sqcup_{k=0}^n d^{-1}(k)$, which has the required properties.

 Let $(S',d',\partial')$ be another $n$-dimensional combinatorial $\Delta$-complex satisfying the conditions~\eqref{E:pres}. Then, we define a map $f\colon S\to S'$ by recursion so that the restriction of $f$  on the subset $d^{-1}\{n,n-1\}$ is the identity, for any $0\le m<n-1$, we set
 \[
 f([\alpha,i])=\partial_i'(f(\alpha)),\quad \alpha\in d^{-1}(m+1),\quad 0\le i\le m+1.
 \]
 It is clear that $f$ is a morphism of combinatorial $\Delta$-complexes, and that it is a unique morphism with the required property.
\end{proof}
By using this proposition, we will call the triple
\[
\langle A,B,\{f_i\colon A\to B\}_{0\le i\le n}\rangle
\]
a \emph{presentation} for the associated universal $n$-dimensional combinatorial $\Delta$-complex.

\subsection{Combinatorial $\Delta$-complexes associated with tangle diagrams}
For an oriented tangle diagram $\Gamma$, we denote by $\Gamma_{04}$ the set of crossings, $\Gamma_{01}$ the set of boundary points, and $\Gamma_1$ the set of edges of $\Gamma$.
 We have two maps
 \[
 \operatorname{dom},\operatorname{cod}\colon \Gamma_1\to \Gamma_{01}\sqcup\Gamma_{04},
 \]
 where $\operatorname{dom}(e)$ (respectively $\operatorname{cod}(e)$) is the crossing or the boundary point where the oriented edge $e$ goes out (respectively goes in). We define another pair of maps
 \[
 \mys,\myt\colon\Gamma_1\to \{0,\pm1\}
 \]
 where $\mys(e)$ (respectively $\myt(e)$) takes the value $+1$, if the edge $e$ is over-passing at the crossing $\operatorname{dom}(e)$ (respectively  at the crossing $\operatorname{cod}(e)$), the value $-1$, if the edge $e$ is under-passing at the crossing $\operatorname{dom}(e)$ (respectively  at the crossing $\operatorname{cod}(e)$), and the value $0$, if $\operatorname{dom}(e)$ (respectively $\operatorname{cod}(e)$) is a boundary point. The function $\varepsilon\colon \Gamma_{04}\to\{\pm1\}$
takes the value $+1$ on positive crossings and $-1$ on negative crossings.

To any tangle diagram $\Gamma$, we associate a $3$-dimensional combinatorial $\Delta$-complex $D_\Gamma$  by the presentation
\[
\langle \Gamma_{04},\Gamma_{1},\{f_i\colon \Gamma_{04}\to \Gamma_{1}\}_{0\le i\le 3}\rangle
\]
where the maps $f_i$ are given by the pictures
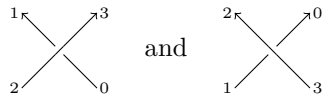
\[
\begin{tikzpicture}[baseline=(x.south)]
 \node (x) at (.5,.5) [inner sep=.5mm,circle] {};
 \draw[->] (x.north east)--(1,1);
 \draw (x.south west)-- (x.north east);
 \draw (0,0)--(x.south west);
 \draw (1,0)--(x.south east);
 \draw[->] (x.north west)--(0,1);
 \node at (1.1,0){\tiny $0$};
 \node at (-.1,0){\tiny $2$};
 \node at (-.1,1){\tiny $1$};
 \node at (1.1,1){\tiny $3$};
 \end{tikzpicture}
 \quad \mathrm{and}\quad
 \begin{tikzpicture}[baseline=(x.south)]
 \node (x) at (.5,.5) [inner sep=.5mm,circle] {};
 \draw[->] (x.north east)--(1,1);
 \draw (x.south east)-- (x.north west);
 \draw (0,0)--(x.south west);
 \draw (1,0)--(x.south east);
 \draw[->] (x.north west)--(0,1);
 \node at (1.1,0){\tiny $3$};
 \node at (-.1,0){\tiny $1$};
 \node at (-.1,1){\tiny $2$};
 \node at (1.1,1){\tiny $0$};
 \end{tikzpicture}
\]
which mean that if, for example, $\varepsilon(v)=+1$, then the left picture tells that $f_0(v)$ is the unique edge $e$ such that $\operatorname{cod}(e)=v$ and $\myt(e)=-1$.

A $\mathbb{Z}$-\emph{coloring} of a $3$-dimensional combinatorial $\Delta$-complex $(S,d,\partial)$ is a map
\[
c\colon d^{-1}(3)\to \mathbb{Z}.
\]

\subsection{Lifting $R$-matrix knot invariants into the combinatorial framework of $\Delta$-complexes}
In the rest of the paper we restrict ourselves to the category $\cfdvs_\mathbb{K}$ of finite dimensional vector spaces over a field $\mathbb{K}$ with the symmetric monoidal structure given by the tensor product over $\mathbb{K}$, and the standard permutation of the tensor components. The unit object is identified with the base field $\mathbb{K}$. A Hopf object in this category is called \emph{finite dimensional Hopf algebra over  $\mathbb{K}$}.

Let $(X,\nabla,\Delta,\eta,\epsilon,\gamma)$ be a finite dimensional Hopf algebra over $\mathbb{K}$ with invertible antipode.
Denote $d=\operatorname{dim}(X)$, and  fix a linear basis $\{ v_i\}_{1\le i\le d}$ in $X$. Let $\{ w^i\}_{1\le i\le d}$ be the corresponding dual basis of the dual vector space $X^*$. We have the canonical dualities
\begin{equation}\label{E:can-du-1}
p=\langle X^*,X,\operatorname{coev},\operatorname{eval}\rangle,\quad \operatorname{eval}(fv)=f(v),\quad
\operatorname{coev}(1)=\sum_{i=1}^d v_iw^i,
\end{equation}
and
\begin{equation}\label{E:can-du-2}
q=\langle X,X^*,\operatorname{coev}',\operatorname{eval}'\rangle,\quad \operatorname{eval}'(vf)=f(v),\quad
\operatorname{coev}'(1)=\sum_{i=1}^d w^iv_i.
\end{equation}
With respect to the chosen basis, the structural maps of our Hopf algebra are given in terms of the corresponding structural constants $\{\nabla_{i,j}^k\}$, $\{\Delta_k^{i,j}\}$, $\{\eta^i\}$, $\{\epsilon_i\}$, $\{\gamma_i^j\}$ according to the formulae
\begin{multline}
\nabla(v_iv_j)=\sum_{k=1}^d\nabla_{i,j}^kv_k,\quad \Delta(v_i)=\sum_{i,j=1}^d\Delta_k^{i,j}v_iv_j,
\\ \eta(1)=\sum_{i=1}^d\eta^iv_i,\quad \epsilon(v_i)=\epsilon_i,\quad \gamma(v_i)=\sum_{j=1}^d\gamma_i^jv_j.
\end{multline}
The action on the basis elements of the maps
\[
\tau_r=X\nabla\circ X\gamma^{\circ r}X\circ\Delta X,\quad r\in\mathbb{Z},
\] has the form
\begin{multline*}
\tau_r v_iv_j=X\nabla\circ X\gamma^{\circ r}X\circ\Delta X(v_iv_j)=\sum_{k,l=1}^d\Delta_i^{k,l}X\nabla\circ X\gamma^{\circ r}X(v_kv_lv_j)\\=
\sum_{k,l,m=1}^d\Delta_i^{k,l}(\gamma^{\circ r})_l^mX\nabla(v_kv_mv_j)
=\sum_{k,l,m,n=1}^d\Delta_i^{k,l}(\gamma^{\circ r})_l^m\nabla_{m,j}^nv_kv_n\\
=\sum_{k,n=1}^d (\tau_r)_{i,j}^{k,n} v_kv_n,
\end{multline*}
where
\[
(\tau_r)_{i,j}^{k,n}=\sum_{l,m=1}^d\Delta_i^{k,l}(\gamma^{\circ r})_l^m\nabla_{m,j}^n\,.
\]

Let $\Gamma$ be a tangle diagram,  $D_\Gamma$, the associated $3$-dimensional combinatorial $\Delta$-complex, and $c\colon \Gamma_{04}\to \mathbb{Z}$, a $\mathbb{Z}$-coloring. For any map
\[
f\colon \Gamma_1\to \{1,\ldots,d\}
\]
we associate a weight function of $W_f(D_\Gamma,c)\in\mathbb{K}$ given by the product
\[
W_f(D_\Gamma,c)=\prod_{v\in\Gamma_{04}} W_f(v,c)
\]
where $W_f(v,c)$ is given by the number $(\tau_{c(v)})_{i,j}^{k,l}$ for the configuration
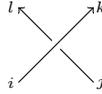
\[
\begin{tikzpicture}[baseline=(x.south)]
 \node (x) at (.5,.5) [inner sep=.5mm,circle] {};
 \draw[->] (x.north east)--(1,1);
 \draw (x.south west)-- (x.north east);
 \draw (0,0)--(x.south west);
 \draw (1,0)--(x.south east);
 \draw[->] (x.north west)--(0,1);
 \node at (1.1,0){\tiny $j$};
 \node at (-.1,0){\tiny $i$};
 \node at (-.1,1){\tiny $l$};
 \node at (1.1,1){\tiny $k$};
 \end{tikzpicture}
\]
and by the number $(\tau_{c(v)})_{i,j}^{k,l}$ for the configuration
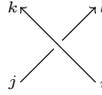
\[
\begin{tikzpicture}[baseline=(x.south)]
 \node (x) at (.5,.5) [inner sep=.5mm,circle] {};
 \draw[->] (x.north east)--(1,1);
 \draw (x.south east)-- (x.north west);
 \draw (0,0)--(x.south west);
 \draw (1,0)--(x.south east);
 \draw[->] (x.north west)--(0,1);
 \node at (1.1,0){\tiny $i$};
 \node at (-.1,0){\tiny $j$};
 \node at (-.1,1){\tiny $k$};
 \node at (1.1,1){\tiny $l$};
 \end{tikzpicture}
\]
The results of the previous sections imply the following proposition.
\begin{proposition}
Let $(X,\nabla,\Delta,\eta,\epsilon,\gamma)$ be a finite dimensional Hopf algebra over $\mathbb{K}$ with invertible antipode, and $\rho_\tau$, the associated rigid $R$-matrix corresponding to the canonical dualities~\eqref{E:can-du-1}, \eqref{E:can-du-2}. Let $\Gamma$ be a tangle diagram, and $\tilde\Gamma$ its $2$-cable where the second component is taken with the reversed orientation. Then, there exists a representative $\Gamma'$ of the dpr-tangle defined by $\tilde \Gamma$, and a $\mathbb{Z}$-coloring $c$ of the $3$-dimensional combinatorial $\Delta$-complex $D_{\Gamma'}$ such that the matrix elements of the operator $\Phi_{\rho_\tau,k}(\Gamma)$ with respect to a chosen basis in $X$ are given by a state sum of the form
\[
\sum_{f,f\vert_{\partial\Gamma'}=g}W_f(D_{\Gamma'},c),
\]
where $g\colon\Gamma'_{01}\to\{1,\ldots,\operatorname{dim}X\}$ is a multi-index parameterizing the matrix elements of the operator $\Phi_{\rho_\tau,k}(\Gamma)$.
\end{proposition}

 \end{document}